\documentclass[reqno,11pt]{amsart}
 
\usepackage{amsmath}
\usepackage{amsthm,amssymb,color,comment}
\usepackage{bbm}
\usepackage{mathrsfs}
\usepackage{hyperref}
\usepackage[TS1,T1]{fontenc}
\usepackage[utf8]{inputenc}
\usepackage{dsfont}
\usepackage{tikz}
\usepackage{enumitem}
\usepackage{subfigure}
\usepackage{mathtools}
\usepackage{bbold}
\usepackage{esint}
\usepackage[sort]{cite}

\numberwithin{equation}{section}

\newtheorem{thm}{Theorem}[section]
\newtheorem{proposition}[thm]{Proposition}
\newtheorem{lem}[thm]{Lemma}

\newtheorem{Def}[thm]{Definition}

\theoremstyle{definition}
\newtheorem{Ass}[thm]{Assumption}
\newtheorem{rem}[thm]{Remark}

\DeclareMathOperator{\DIV}{div}

\newcommand{\R}{\mathbb{R}}

\newcommand{\N}{\mathbb{N}}
\newcommand{\p}{\partial}

\newcommand{\eps}{\varepsilon}

\newcommand{\diff}{\mathop{}\!\mathrm{d}}

\newcommand{\weaks}{\overset{\ast}{\rightharpoonup}}

\newcommand{\weak}{\rightharpoonup}

\makeatletter
\newcommand{\doublewidetilde}[1]{{%
  \mathpalette\double@widetilde{#1}%
}}
\newcommand{\double@widetilde}[2]{%
  \sbox\z@{$\m@th#1\widetilde{#2}$}%
  \ht\z@=.9\ht\z@
  \widetilde{\box\z@}%
}
\makeatother
\author{Charles Elbar}
\address{{\it Charles Elbar:} Sorbonne Universit\'{e}, Laboratoire Jacques-Louis Lions (LJLL), F-75005 Paris, France}
\email{charles.elbar@sorbonne-universite.fr}
\thanks{}

\author{Jakub Skrzeczkowski}
\address{{\it Jakub Skrzeczkowski: } Faculty of Mathematics, Informatics and Mechanics, University of Warsaw, Poland}
\email{jakub.skrzeczkowski@student.uw.edu.pl}
\thanks{Jakub Skrzeczkowski was supported by National Science Center, Poland through project no. 2021/43/B/ST1/02851. Authors would like to thank Noemi David, Tomasz Dębiec and Markus Schmidtchen for a fruitful discussion on the topic of this paper.}

\begin{document}

\title[On the inviscid limit connecting Brinkman's and Darcy's models]{On the inviscid limit connecting Brinkman's and Darcy's models of tissue growth with nonlinear pressure}

\begin{abstract}
Several recent papers have addressed modelling of the tissue growth by the multi-phase models where the velocity is related to the pressure by one of the physical laws (Stoke's, Brinkman's or Darcy's). While each of these models has been extensively studied, not so much is known about the connection between them. In the recent paper (arXiv:2303.10620), assuming the linear form of the pressure, the Authors connected two multi-phase models by an inviscid limit: the viscoelastic  one (of Brinkman's type) and the inviscid one (of Darcy's type). Here, we prove that the same is true for a nonlinear, power-law pressure. The new ingredient is that we use relation between the pressure $p$ and the Brinkman potential $W$ to deduce compactness in space of $p$ from the compactness in space of~$W$.

\end{abstract}



\maketitle
\setcounter{tocdepth}{1}

\section{Introduction}
Last years brought deep understanding of mechanical models of tissue growth. These models are based on the continuity equation for the density $\rho$
$$
\partial_t \rho + \DIV(\rho v) = 0,
$$
where the velocity $v$ is linked to the pressure $p$ which is assumed to be a power-law function of density $\rho$ i.e. $p(\rho) = \rho^{\gamma}$ for some $\gamma \geq 1$. The most widely studied one, Darcy's law, asserts that the velocity $v = -\nabla p$. Such approach has been thoroughly studied 
\cite{david2022convergence,MR3162474,MR3292119,MR4367913,MR4215195}, also in the context of two populations \cite{MR4072681,MR4000848,MR3795211,Liuguo,MR3579573,MR4179253,MR2952633}, presence of a nutrient \cite{MR3260280,MR4324293}, more general Patlak-Keller-Segel equation \cite{he2022incompressible} or additional advection effects \cite{david2021incompressible, MR3942711}. Another approach is the
Brinkman's law \cite{MR4146915,DEBIEC2021204,Perthame-incompressible-visco, MR3037041, MR3813966}. Here, velocity equals  $v = -\nabla W$ where $W$ solves an elliptic equation
$$
-\sigma \Delta W + W = p,
$$
for some small $\sigma > 0$. In this approach, the velocity enjoys higher regularity due to elliptic regularity theory. Last but not least, one can consider the Stoke's law where the velocity is given by the corresponding Navier-Stokes(-Korteweg) equation \cite{MR3695967,MR1687274,MR3119735} or with additional surface-tension effects \cite{elbar-perthame-skrzeczkowski, degond2022multi,elbar2021degenerate,elbar2023nonlocal} by including Cahn-Hilliard-type terms.\\

Most of the studies discussed above have been carried out to in the context of so-called incompressible limit. This procedure links mechanistic models and free-boundary problems extensively used in the context of tumor growth \cite{byrne1996growth}. Mathematically, the limit corresponds to sending $\gamma \to \infty$ in the pressure relation $p(\rho) = \rho^{\gamma}$. In the limit, $\rho \leq 1$ and the zone $\Omega_t := \{x: \rho(t,x)=1\}$ is interpreted as a tumor resulting in the free boundary problem which has been extensively studied, see for instance \cite{MR4587548,MR3485146,MR3695889, MR3729490}. In this context, it is worth mentioning another form of the pressure $p(\rho) = \varepsilon \, \frac{\rho}{1-\rho}$ which enforces the density to stay below 1 so it is useful for modeling populations with congestion constraints, see \cite{MR3717913, MR4063893, MR3974475,MR3258257,MR3355504}. For such pressure laws, one can also study incompressible limit by sending $\varepsilon \to 0$. \\

In the present work, we are interested in linking the two populations model of Brinkman's type with the one of Darcy's type. Hence, we consider the system of PDEs posed on $[0,T]\times\R^d$ 
\begin{equation}\label{eq:Brinkman}
\partial_t u_{\sigma} - \DIV (u_{\sigma} \nabla W_{\sigma}) = u_{\sigma}\, F(p_{\sigma}), \qquad 
\partial_t v_{\sigma} - \DIV (v_{\sigma} \nabla W_{\sigma}) = v_{\sigma}\, G(p_{\sigma}),
\end{equation}
where $u_{\sigma}$, $v_{\sigma}$ are densities of two populations of interest, $p = (u_{\sigma}+v_{\sigma})^{\gamma}$ is the pressure, $\gamma > 1$ and $W_{\sigma}$ is the solution of the elliptic equation
\begin{equation}\label{eq:elliptic_PDE}
-\sigma \Delta W_{\sigma} + W_{\sigma} = p_{\sigma}  
\end{equation}
corresponding to the so-called Brinkman's law. Our target is to rigorously justify the limit $\sigma \to 0$ where we expect the Darcy's law $W = p = (u+v)^\gamma$ and the densities $u$, $v$ satisfy
\begin{equation}\label{eq:Darcy}
\partial_t u - \DIV (u \nabla p) = u\, F(p), \qquad 
\partial_t v - \DIV (v \nabla p) = v\, G(p).
\end{equation}
In \cite{david2023degenerate}, the case of $\gamma = 1$ was established. Here, we study the nonlinear case $\gamma > 1$. As discussed above, from the point of view of free boundary models, large $\gamma$ is more physically relevant and this motivates our studies.\\

We first list the assumptions which are standard in the theory of \eqref{eq:Brinkman}. 
\begin{Ass}\label{ass_main} We assume that: 
\begin{enumerate}[label=(\Alph*)]
    \item \label{ass:item1}The nonlinearities $F$, $G$ belong to $C^1(\R)$ and they are strictly decreasing: $F', G' \leq -\alpha < 0$ for some $\alpha > 0$. Moreover, there exists $p_H>0$ (the so-called homeostatic pressure) such that  
    $F(p_H) = G(p_H) = 0$. 
    \item The initial condition $(u^0, v^0)$ is nonnegative and satisfies the following: the upper bound $(u^0 + v^0)^{\gamma} \leq p_H$, the mass bound $\int_{\R^d} (u^0 + v^0) \diff x \leq C$ and the tail estimate $\int_{\R^d} (u^0 + v^0) |x|^2 \diff x \leq C$.
\end{enumerate}
\end{Ass}

The weak solutions to the systems \eqref{eq:Brinkman}-\eqref{eq:elliptic_PDE} and \eqref{eq:Darcy} are defined as follows:

\begin{Def}[Weak solutions to the Brinkman system]\label{def:Brinkman}
We say that $(u_{\sigma}, v_{\sigma})$ is a weak solution of~\eqref{eq:Brinkman}-\eqref{eq:elliptic_PDE} with initial condition $(u^0, v^0)$ if $u_\sigma, v_\sigma \in L^{\infty}(0,T; L^{1}(\R^d)\cap L^{\infty}(\R^d))$ and for all $\varphi \in C_{c}^{\infty}([0,T)\times \R^d)$ and $\phi \in C_{c}^{\infty}([0,T)\times \R^d)$: 
\begin{multline*}
\int_{0}^{T}\int_{\R^{d}} u_{\sigma} \, \p_{t}\varphi \diff x\diff t  + \int_{\R^{d}} \varphi(0,x) u^{0}(x) \diff x = \\ = \int_{0}^{T}\int_{\R^{d}} u_{\sigma}\nabla W_{\sigma}\cdot\nabla\varphi\diff x \diff t -  \int_{0}^{T}\int_{\R^{d}} u_{\sigma} F(p_{\sigma})\varphi \diff x\diff t,
\end{multline*}
\begin{multline*}
\int_{0}^{T}\int_{\R^{d}} v_{\sigma} \, \p_{t}\phi \diff x\diff t  + \int_{\R^{d}} \phi(0,x) v^{0}(x) \diff x = \\ =  \int_{0}^{T}\int_{\R^{d}} v_{\sigma}\nabla W_{\sigma}\cdot\nabla\phi\diff x \diff t - \int_{0}^{T}\int_{\R^{d}} v_{\sigma} G(p_{\sigma})\phi\diff x\diff t,
\end{multline*}
with $p_{\sigma}= (u_{\sigma}+ v_{\sigma})^{\gamma}$ and 
$
-\Delta W_{\sigma} + W_{\sigma} = p_{\sigma}
$
a.e. in $(0,T)\times \R^{d}$.
\end{Def}
We note that the terms $u_{\sigma}\,\nabla W_{\sigma}$, $v_{\sigma}\,\nabla W_{\sigma}$ make sense. Indeed, one can write $W_{\sigma} = K_{\sigma} \ast p_{\sigma}$ where $K_{\sigma}$ is a fundamental solution of $-\sigma \Delta K_{\sigma} + K_{\sigma} = \delta_0$. It is well-known (see, for instance, \cite[eq. (2.6)]{Perthame-incompressible-visco}) that $K_{\sigma} \geq 0$, $\int_{\R^d} K_{\sigma} \diff x = 1$ and $\nabla K_{\sigma} \in L^1(\R^d)$ where the last estimate blows up when $\sigma \to 0$. Therefore, for $\sigma$ fixed, $\nabla W_{\sigma} \in L^{\infty}(0,T; L^q(\R^d))$ for all $q \in [1,\infty]$. 
\begin{Def}[Weak solutions to the Darcy system]\label{def:Darcy}
We say that $(u, v)$ is a weak solution of~\eqref{eq:Darcy} with initial condition $(u^0, v^0)$ and $p = (u+v)^{\gamma}$ if $u, v \in L^{\infty}(0,T; L^{1}(\R^d)\cap L^{\infty}(\R^d))$, $\nabla p \in L^2((0,T)\times\R^d)$ and for all $\varphi \in C_{c}^{\infty}([0,T)\times \R^d)$ and $\phi \in C_{c}^{\infty}([0,T)\times \R^d)$: 
\begin{multline*}
\int_{0}^{T}\int_{\R^{d}} u \, \p_{t}\varphi \diff x\diff t  + \int_{\R^{d}} \varphi(0,x) u^{0}(x) \diff x = \\ = \int_{0}^{T}\int_{\R^{d}} u\nabla p\cdot\nabla\varphi\diff x \diff t -  \int_{0}^{T}\int_{\R^{d}} u F(p)\varphi\diff x\diff t,
\end{multline*}
\begin{multline*}
\int_{0}^{T}\int_{\R^{d}} v \, \p_{t}\phi \diff x\diff t  + \int_{\R^{d}} \phi(0,x) v^{0}(x) \diff x = \\ =  \int_{0}^{T}\int_{\R^{d}} v \nabla p\cdot\nabla\phi\diff x \diff t - \int_{0}^{T}\int_{\R^{d}} v G(p)\phi\diff x\diff t.
\end{multline*}
\end{Def}

The existence of weak solutions to the Brinkman system is given by the following result.
\begin{thm}\label{thm_existence_of_solutions}
Under Assumption \ref{ass_main}, there exists a weak solution $(u_{\sigma}, v_{\sigma})$ to the system~\eqref{eq:Brinkman}-\eqref{eq:elliptic_PDE} in the sense of Definition \ref{def:Brinkman}. Moreover, the solution is uniformly bounded
\begin{equation}\label{eq:uniform_bounds_existence_thm}
0\leq (u_{\sigma} + v_{\sigma})^{\gamma} \leq p_H.
\end{equation}
\end{thm}

The existence result is  fairly standard and it is based on suitable regularizations. Nevertheless, it contains few interesting technical difficulties, therefore we present the proof in the Appendix \ref{sect:existence_solutions}. Let us point that that the uniform bound \eqref{eq:uniform_bounds_existence_thm} is the direct consequence of the maximum principle in \cite[Lemma 2.1]{MR3037041}.\\

Our main result is the rigorous justification of passing to the limit $\sigma \to 0$. 
\begin{thm}[Brinkman to Darcy]\label{thm:brinkman_to_darcy}
Let $(u_{\sigma}, v_{\sigma})$ be a weak solution of system~\eqref{eq:Brinkman}-\eqref{eq:elliptic_PDE} as in Theorem \ref{thm_existence_of_solutions}. Then, as $\sigma\to 0$, we can extract a subsequence (not relabeled) such that
\begin{align*}
& u_\sigma \to u \quad \text{weakly* in $L^{\infty}((0,T)\times\R^d )$ and weakly in $L^q((0,T)\times\R^d)$ for $1\leq q<\infty$},\\
& v_\sigma \to v \quad \text{weakly* in $L^{\infty}((0,T)\times\R^d )$ and weakly in $L^q((0,T)\times\R^d)$ for $1\leq q<\infty$}, \\
& p_\sigma \to p \quad \text{strongly in  $L^{q}((0,T)\times \R^{d})$ for all $1\le q < +\infty$},\\
& W_\sigma \to p \quad \text{ strongly in $L^{2}(0,T; H^{1}(\R^{d}))\cap L^{q}((0,T)\times \R^{d})$ for all $1 < q< +\infty$},
\end{align*}
with $p=(u+v)^{\gamma}$. Moreover $(u,v)$ is a weak solution of \eqref{eq:Darcy} as in Definition \ref{def:Darcy}.
\end{thm}

Let us briefly outline the strategy. The main difficulty is to pass to the limit in the terms $u_{\sigma}\, \nabla W_{\sigma}$ and $v_{\sigma}\, \nabla W_{\sigma}$. It seems that there is no hope for the strong compactness of $u_{\sigma}$ and $v_{\sigma}$ because this requires at least uniform bounds on $\{D^2 W_{\sigma}\}$ (cf. \cite{MR2995703}) which does not seem to be available (see \eqref{eq:mass_pressure} for the energy identity). Therefore, we plan to prove strong compactness of $\{\nabla W_{\sigma}\}$ in $L^2((0,T)\times\R^d)$ by proving weak compactness and convergences of norms:
$$
\nabla W_{\sigma} \weak \nabla p \mbox{ in } L^2((0,T)\times \R^d), \qquad \lim_{\sigma \to 0} \|\nabla W_{\sigma}\|^2_{L^2((0,T)\times\R^d)} = \|\nabla p\|^2_{L^2((0,T)\times \R^d)},
$$
which was recently applied in several problems of similar nature \cite{jacobs2021existence, david-phenotypic, Liuguo, MR4179253}. This can be achieved if one proves strong compactness of $\{p_{\sigma}\}$. Indeed, weak compactness of $\{\nabla W_{\sigma}\}$ is then a consequence of the energy estimate \eqref{eq:mass_pressure} and the elliptic equation \eqref{eq:elliptic_PDE}. The convergence of norms follows from the energy: we compare energy \eqref{eq:mass_pressure} for $\sigma \to 0$ with the energy for 
$$
\partial_t (u+v) - \DIV((u+v) \, \nabla p) = u\,F(p) + v\,G(p),
$$
which can be written because $\{p_{\sigma}\}$ is strongly compact.\\

It remains to explain how we obtain strong compactness of the pressures $\{p_{\sigma}\}$. From a priori estimates (Proposition \ref{prop:unif_bounds}) we know that $\{p_{\sigma}\}$ is compact in time while $\{W_{\sigma}\}$ is compact in space. Moreover, the term $\sigma \Delta W_{\sigma}$ converges strongly to 0. Therefore, we can use the elliptic equation \eqref{eq:elliptic_PDE} to translate information about compactness of $\{W_{\sigma}\}$ into compactness of $\{p_{\sigma}\}$. Details are given in Lemma \ref{lem:Jakub} and Lemma \ref{lem:compact_pressure}.\\

We also remark that our method covers the linear case $\gamma= 1$ studied in \cite{david2023degenerate}. The only difference is that the energy identity used to obtain all the estimates and deduce strong compactness of $\{\nabla W_{\sigma}\}$ is deduced by multiplying equation for the sum $u_{\sigma}+v_{\sigma}$ with $\log(u_{\sigma}+v_{\sigma})$. Some care is necessary as this function may not be admissible in the vacuum where $u_{\sigma}+v_{\sigma} = 0$ and the details are discussed in \cite{david2023degenerate}.\\

Finally, let us remark that the problem of passing to the limit from \eqref{eq:Brinkman} to \eqref{eq:Darcy} can be seen in a much broader context of passing to the limit from the nonlocal equation to the local one. More precisely, \eqref{eq:elliptic_PDE} can be written as $W_{\sigma} = K_{\sigma} \ast p_{\sigma}$ where $K_{\sigma}$ is a kernel approaching Dirac mass $\delta_0$ so that $W = p$ in the limit $\sigma \to 0$. Such problems are intesively studied for several PDEs, including porous media equation \cite{carrillo2023nonlocal, burger2022porous, MR1821479,Hecht2023porous,MR3913840}, Cahn-Hilliard equation (both nondegenerate \cite{MR4198717,
MR4093616} and degenerate \cite{elbar-skrzeczkowski,carrillo2023degenerate,elbar2023limit}) and hyperbolic conservation laws \cite{coclite2023ole}.

\section{A priori estimates}

Here, we prove the following:
\begin{proposition}\label{prop:unif_bounds}
Let $\sigma > 0$. Let $(u_{\sigma}, v_{\sigma})$ be a weak solution of \eqref{eq:Brinkman}--\eqref{eq:elliptic_PDE}. Then,  the following sequences are uniformly bounded with respect to $\sigma \in (0,1)$:
\begin{enumerate}[label=(\Alph*)]
    \item $\{p_{\sigma}\}$ and $\{W_{\sigma}\}$ in $L^{\infty}(0,T; L^1(\R^d) \cap L^{\infty}(\R^d))$\label{estim:A},
    \item \label{estim:B} $\{\nabla W_{\sigma}\}$ in $L^2((0,T)\times\R^d)$,
    \item $\{\sqrt{\sigma}\,\Delta W_{\sigma}\}$ in $L^2((0,T)\times\R^d)$ \label{estim:C},
    \item $\{\partial_t p_{\sigma}\}$ in $L^{1}(0,T; H^{-s}_{loc}(\R^{d}))$, \, for $s$ large enough, \label{estim:D} 
    \item $\{p_{\sigma}\, |x|^{2}\}$ in $L^{\infty}(0,T; L^{1}(\R^{d}))$, \label{estim:E}
\end{enumerate}
\end{proposition}
\begin{proof}
First, the $L^{\infty}$ estimate for $p_{\sigma}$ is a direct consequence of \eqref{eq:uniform_bounds_existence_thm} and the same is true for $W_{\sigma} = K_{\sigma} \ast p_{\sigma}$ because $\int_{\R^d} K_{\sigma} \diff x = 1$. The estimates~\ref{estim:A}, \ref{estim:B} and~\ref{estim:C} are a consequence of energy considerations. The equation for the sum $u_{\sigma}+v_{\sigma}$ reads:
\begin{equation}\label{eq:sum}
\partial_t(u_{\sigma}+v_{\sigma}) - \DIV((u_{\sigma}+v_{\sigma})\nabla W_{\sigma}) = u_{\sigma}\, F(p_{\sigma}) + v_{\sigma}\,G(p_{\sigma}).
\end{equation}
We multiply with $\gamma\, (u_{\sigma}+v_{\sigma})^{\gamma-1}$ and integrate in space so that
\begin{multline*}
\partial_t \int_{\R^d} p_{\sigma} \diff x + \gamma\, ( \gamma-1)\int_{\R^d}  (u_{\sigma}+v_{\sigma})^{\gamma-1}\, \nabla (u_\sigma + v_\sigma) \cdot \nabla W_{\sigma}\diff x = \\ = \gamma \int_{\R^d} (u_{\sigma}\,F(p_{\sigma}) + v_{\sigma}\, G(p_{\sigma})) (u_{\sigma}+v_{\sigma})^{\gamma-1} \diff x,
\end{multline*}
which, with~\eqref{eq:elliptic_PDE}, can be rewritten as
\begin{equation}\label{eq:mass_pressure}
\begin{split}
\partial_t \int_{\R^d} p_{\sigma} \diff x &+  ( \gamma-1)\int_{\R^d}  |\nabla W_{\sigma}|^2 + \sigma |\Delta W_{\sigma}|^2 \diff x \\&= \gamma \int_{\R^d} (u_{\sigma}\,F(p_{\sigma}) + v_{\sigma}\,G(p_{\sigma})) (u_{\sigma}+v_{\sigma})^{\gamma-1} \diff x \leq C \, \int_{\R^d} p_{\sigma} \diff x,
\end{split}
\end{equation}
because $p_{\sigma}$ is uniformly bounded and $F$, $G$ are continuous. Finally, we have  $\|W_{\sigma}\|_{L^1(\R^d)} = \|K_{\sigma}\ast p_{\sigma}\|_{L^1(\R^d)} = \|p_{\sigma}\|_{L^1(\R^d)}$ which concludes the proof of~\ref{estim:A}, \ref{estim:B} and~\ref{estim:C}. \\
Next, we establish the bound~\ref{estim:D} on $\partial_t p_{\sigma}$. We first write the equation on $p$ which can be obtained after multiplying~\eqref{eq:sum} by $\gamma (u_\sigma + v_\sigma)^{\gamma-1} $:
\begin{equation*}
\p_{t} p_{\sigma} =  \nabla p_{\sigma}\nabla W_{\sigma} + \gamma\, p_{\sigma}\, \Delta W_{\sigma} +  \gamma\left(u_{\sigma} F(p_\sigma) + v_\sigma G(p_\sigma) \right)(u_\sigma + v_\sigma )^{\gamma-1}. 
\end{equation*}
First note that the last term on the right-hand side is bounded in $L^{\infty}((0,T)\times \R^d)$ by assumptions on $F$ and $G$. In order to obtain a bound on $\p_t p_{\sigma}$ in a negative Sobolev spaces, we see that, up to integration by parts, it remains to study the term $p_{\sigma}\Delta W_{\sigma}$. Let $\varphi$ be a smooth, compactly supported test function. Then, by definition of $W_{\sigma}$ 
\begin{equation*}
\int_{\R^{d}} p_{\sigma} \Delta W_{\sigma} \varphi \diff x = -\sigma\int_{\R^{d}} |\Delta W_{\sigma}|^2 \varphi\diff x - \int_{\R^d} |\nabla W_{\sigma}|^{2}\varphi \diff x - \int_{\R^{d}} W_{\sigma} \nabla W_{\sigma}\cdot \nabla \varphi \diff x.
\end{equation*}

The proof of~\ref{estim:D} is concluded using \ref{estim:A}, \ref{estim:B} and \ref{estim:C}. Now, we prove~\ref{estim:E}. Since $p_{\sigma}$ is bounded in $L^{\infty}$ we only need to prove that $u_\sigma$ and $v_\sigma$ have uniformly bounded second moments. We compute it for $u_\sigma$ and the proof is similar for $v_\sigma$
$$
\p_t \int_{\R^d} |x|^{2} u_{\sigma} \diff x + 2\int_{\R^d}u_{\sigma}\nabla W_{\sigma}\cdot x\diff x = \int_{\R^d}|x|^{2} u_{\sigma} F(p_{\sigma}) \diff x. 
$$
Integrating in time, using Cauchy-Schwartz inequality, estimate \ref{estim:B}, assumptions on $F$ and Gronwall's inequality, we obtain the result.
\end{proof}

\section{Strong compactness of the pressure}


\begin{proposition}\label{prop:convergence}
There exist functions $u$, $v$ and $p$ such that $p= (u+v)^{\gamma}$ a.e. and such that up to a subsequence (not relabelled) for all $1 \leq q < \infty$:
\begin{align}
& u_\sigma \to u \quad \text{weakly* in $L^{\infty}((0,T)\times\R^d )$ and weakly in $L^{q}((0,T)\times\R^d)$}  ,\label{weak_u}\\
& v_\sigma \to v \quad \text{weakly* in $L^{\infty}((0,T)\times\R^d )$ and weakly in $L^{q}((0,T)\times\R^d)$}, \label{weak_v}\\
& \sigma \Delta W_\sigma \to 0 \quad \text{strongly in $L^{2}((0,T)\times \R^d)$},\label{weak_laplace}\\
& p_\sigma \to p \quad \text{strongly in  $L^{q}((0,T)\times \R^{d})$},\label{strong_p}\\
& W_\sigma \to p \quad \text{weakly in $L^{2}(0,T; H^{1}(\R^{d}))$, strongly in $L^{q}((0,T)\times \R^{d})$ for $q>1$},\label{strong_W}\\
& u_\sigma + v_\sigma \to u+v \quad \text{strongly in $L^{q}((0,T)\times \R^{d})$}\label{strong_sum},\\
& (u_\sigma + v_\sigma)(T) \to (u+v)(T) \quad \text{weakly in } L^{\gamma}(\R^d). \label{eq:weak_conv_trace}
\end{align}
\end{proposition}

The crucial step in the proof of Proposition \ref{prop:convergence} is the strong compactness of the pressure $p_{\sigma}$ which will be achieved by the following lemma which in the spirit is a version of Lions-Aubin-Simon's argument \cite{MR916688}.
\begin{lem}
\label{lem:Jakub}
Suppose that for each compact set $K \subset \R^d$
\begin{equation}\label{eq:compactness_in_space}
\lim_{y \to 0} \int_0^T \int_{K} |p_{\sigma}(t,x+y)-p_{\sigma}(t,x)| \diff x \diff t = 0 \mbox{ uniformly in } \sigma \in (0,1).
\end{equation}
Moreover, assume that $\{\partial_t p_{\sigma}\}$ is bounded in $L^1(0,T; H^{-s}_{\text{loc}}(\R^d))$ for some $s>0$ and $\{p_{\sigma}\, |x|^2\}$ is uniformly bounded in $L^1((0,T)\times\R^d)$. Then, the sequence $\{p_{\sigma}\}$ is strongly compact in $L^1((0,T)\times\R^d)$.
\end{lem}
\begin{rem}
Several variants of this result are possible. For instance, one can have more general assumption on the time derivative or one can also formulate it in for the space $L^p((0,T)\times \R^d)$ space with $p > 1$. Another trivial generalization is that the tail estimate could be replaced with more general tightness assumption.
\end{rem}
\begin{proof}[Proof of Lemma \ref{lem:Jakub}]
In view of the Riesz-Kolmogorov-Frechet theorem, to establish local compactness it is sufficient to prove
$$
\lim_{h \to 0} \int_0^{T-h} \int_{K} |p_{\sigma}(t+h,x)-p_{\sigma}(t,x)| \diff x \diff t = 0 \mbox{ uniformly in } \sigma \in (0,1)
$$
for each compact set $K \subset \R^d$. Using a family of smooth, compactly supported mollifiers $\{\varphi_{\delta}\}$ with $\delta$ depending on $h$, we have 
\begin{align*}
    \int_0^{T-h} \int_{K} |p_{\sigma}(t+h,x)&-p_{\sigma}(t,x)| \diff x \diff t \leq \\
    \leq &\, \int_0^{T-h} \int_{K} |p_{\sigma}(t+h,x)-p_{\sigma}(t+h,x)\ast \varphi_{\delta}| \diff x \diff t \\
    &+  \int_0^{T-h} \int_{K} |p_{\sigma}(t,x)-p_{\sigma}(t,x)\ast \varphi_{\delta}| \diff x \diff t\\
    &+  \int_0^{T-h} \int_{K} |p_{\sigma} \ast \varphi_{\delta}(t+h,x)-p_{\sigma} \ast \varphi_{\delta}(t,x)| \diff x \diff t.
\end{align*}
The first two terms converge to 0 when $\delta \to 0$, independently of $h$, as a consequence of \eqref{eq:compactness_in_space}. Hence, we only need to study the third term. We write 
$$
p_{\sigma} \ast \varphi_{\delta}(t+h,x)-p_{\sigma} \ast \varphi_{\delta}(t,x) = h\,\int_0^1 \partial_t p_{\sigma} \ast \varphi_{\delta}(t+s\,h,x) \diff s.
$$
Therefore, the term of interest can be estimated by
\begin{multline*}
C\, h \int_0^{T-h} \int_{K} \left| \int_0^1 \partial_t p_{\sigma} \ast \varphi_{\delta}(t+s\,h,x) \diff s \right| \diff x \diff t \leq \\ \leq 
C\, h \int_0^1 \int_0^{T-h} \int_{K} \left|  \partial_t p_{\sigma} \ast \varphi_{\delta}(t+s\,h,x)  \right| \diff x \diff t \diff s \leq C\, h \| \partial_t p_{\sigma} \ast \varphi_{\delta} \|_{L^1((0,T)\times K)},
\end{multline*}
where we applied Fubini's theorem. It remains to estimate the convolution. We have
$$
\partial_t p_{\sigma} \ast \varphi_{\delta}(t,x) = \int_{\R^d} \partial_t p_{\sigma}(t,y) \, \varphi_{\delta}(x-y) \diff y \leq \| \partial_t p_{\sigma}(t,\cdot) \|_{H^{-s}_{loc}} \, \| \varphi_{\delta} \|_{H^{s}},
$$
Applying the $L^1((0,T)\times K)$ norm we obtain
$$
\|\partial_t p_{\sigma} \ast \varphi_{\delta}\|_{L^1((0,T)\times K)} \leq |K| \, \| \partial_t p_{\sigma} \|_{L^1_t H^{-s}_{loc,x}} \, \| \varphi_{\delta} \|_{H^{s}} \leq \frac{C(K)}{\delta^{s+d/2}}.
$$
Choosing $h = \delta^{1+s+d/2}$ we obtain compactness of $\{p_{\sigma}\}$ on $(0,T)\times K$ for each compact set $K$.\\

To obtain global compactness, we perform a usual argument which uses the decay estimate. Let $B_n$ be the sequence of balls $B(0,n)$. By the diagonal method, we construct a subsequence such that $p_{\sigma} \to p$ in $L^1((0,T)\times B_n)$ for each $n \in \N$. Then,
\begin{align*}
\|p_{\sigma} - p\|_{L^1((0,T)\times\R^d)} &\leq \|p_{\sigma} - p\|_{L^1((0,T)\times B_n)} + \|p_{\sigma} - p\|_{L^1((0,T)\times (\R^d\setminus B_n))} \\
&\leq \|p_{\sigma} - p\|_{L^1((0,T)\times B_n)} + \frac{C}{n^2},
\end{align*}
where $C = \sup_{\sigma\in(0,1)} \| p_{\sigma} \, |x|^2\|_{L^1((0,T)\times\R^d)}$ and we used, that by Fatou lemma, 
$$
\|p |x|^2\|_{L^1((0,T)\times\R^d)} \leq \liminf_{\sigma \to 0} \| p_{\sigma} \, |x|^2\|_{L^1((0,T)\times\R^d)}.
$$
Hence, $\limsup_{\sigma \to 0} \|p_{\sigma} - p\|_{L^1((0,T)\times\R^d)} \leq \frac{C}{n^2}$ for all $n$ which concludes the proof.
\end{proof}

By interpolation, we deduce:
\begin{lem}\label{lem:compact_pressure}
The sequence of pressures $\{p_\sigma\}$ is strongly compact in $L^{q}((0,T)\times\R^{d})$ for all $1\le q<+ \infty$.    
\end{lem}

It turns out that the sequence of pressures $\{p_\sigma\}$ satisfies~\eqref{eq:compactness_in_space}. 

\begin{proof}[Proof of~\eqref{eq:compactness_in_space}]
We recall the definition of the Brinkman law 
$$
p_{\sigma}= -\sigma \Delta W_{\sigma} + W_{\sigma}. 
$$

First note that $-\sigma\Delta W_{\sigma}$ converges strongly to 0 in $L^{2}((0,T)\times\R^{d})$ as a consequence of~\ref{estim:C}. In particular it is compact and by the converse of the Riesz-Kolmogorov-Frechet theorem we deduce 
\begin{equation*}
\lim_{y \to 0} \int_0^T \int_{K} |\sigma\Delta W_{\sigma}(t,x+y)-\sigma\Delta W_{\sigma}(t,x)| \diff x \diff t = 0 \mbox{ uniformly in } \sigma \in (0,1).
\end{equation*}
To conclude the proof of \eqref{eq:compactness_in_space} in remains to prove that 
\begin{equation*}
\lim_{y \to 0} \int_0^T \int_{K} | W_{\sigma}(t,x+y)- W_{\sigma}(t,x)| \diff x \diff t = 0 \mbox{ uniformly in } \sigma \in (0,1).
\end{equation*}
This is a simple consequence of the formula
$$
W_{\sigma}(t,x+y)- W_{\sigma}(t,x)=\int_{0}^{1}\nabla W_{\sigma}(t, x+ sy)\cdot y \diff s
$$
and the uniform bound \ref{estim:B} (note that we work on the compact set $K$ so that $L^2(K)$ embeds into $L^1(K)$).

\end{proof}

\begin{proof}[Proof of Proposition \ref{prop:convergence}]
Convergences~\eqref{weak_u}, \eqref{weak_v}, \eqref{weak_laplace} follow from Proposition~\ref{prop:unif_bounds} and nonnegativity of $u_{\sigma}$, $v_{\sigma}$. The strong convergence of the pressure~\eqref{strong_p} is a consequence of Lemma \ref{lem:compact_pressure}. Combining this convergence and~\eqref{weak_laplace} we deduce the strong convergence in $L^{2}$ of $W_{\sigma}$ and then in every $L^{q}$ (except $q=1, \infty$) by interpolation using~\ref{estim:A}. Using also the estimate~\ref{estim:B}, we deduce the weak convergence of $\nabla W_{\sigma}$ and conclude the proof of~\eqref{strong_W}.
The convergence~\eqref{strong_sum} is a consequence of~\eqref{strong_p}: indeed we can extract a subsequence of pressures that converge a.e. so that $u_{\sigma}+v_{\sigma}$ converges a.e. Due to the uniform $L^{\infty}$ bound \ref{estim:A} and tail estimate \ref{estim:E}, Vitali convergence theorem implies~\eqref{strong_sum}.\\

Finally, we prove \eqref{eq:weak_conv_trace}. We adapt the argument from \cite[Lemma II.5.9]{Boyer}. First, we prove that $(u_{\sigma}+v_{\sigma})(T)$ makes sense as an element of $L^{\gamma}(\R^d)$. Let $\eta_{\delta} = \frac{1}{\delta}\mathds{1}_{[-\delta,0]}$ and $f_{\delta}(t,x) = (u_{\sigma}+v_{\sigma}) \ast \eta_{\delta}$ where the convolution is a convolution in time. As $u_{\sigma}+v_{\sigma} \in L^{\infty}(0,T; L^{\gamma}(\R^d))$, $\|f_{\delta}(T,\cdot)\|_{L^{\gamma}(\R^d)} \leq C$ independently of $\delta>0$. Hence, up to a subsequence, $f_{\delta}(T) \weak f$ in $L^{\gamma}(\R^d)$ and it remains to prove $f = (u_{\sigma}+v_{\sigma})(T)$. Let $\psi \in C_c^{\infty}(\R^d)$. Since 
\begin{equation}\label{eq:functionals_sigma_fixed_T}
\int_{\R^d} f_{\delta}(T,x) \, \psi(x) \diff x = \left[\int_{\R^d} (u_{\sigma} + v_{\sigma})(\cdot,x)\, \psi(x) \diff x\right] \ast \eta_{\delta}(T)
\end{equation}
and the function $t \mapsto \int_{\R^d} (u_{\sigma}+v_{\sigma})\, \psi(x) \diff x$ is continuous (it can be easily seen that the sequence $\{\partial_t (u_{\sigma}+v_{\sigma})\}$ is uniformly bounded in $L^2(0,T; H^{-1}_{loc}(\R^d))$ and such regularity implies also continuity in $C(0,T; H^{-1}_{loc}(\R^d))$, see \cite[Lemma 7.1]{MR3014456}), the (RHS) of \eqref{eq:functionals_sigma_fixed_T} converges to $\int_{\R^d} (u_{\sigma}(T) + v_{\sigma}(T))\, \psi(x) \diff x$ so that $f=(u_{\sigma} + v_{\sigma})(T)$ a.e. on $\R^d$. Exactly the same argument shows that $(u+v)(T)$ makes sense because, thanks to convergences \eqref{weak_u}, \eqref{weak_v}, \eqref{strong_W}, \eqref{strong_sum}, we can pass to the limit $\sigma\to0$ and deduce the same weak formulation
\begin{equation}\label{eq:weak_form_the_limit_sum}
\begin{split}
\int_{0}^{T}\int_{\R^{d}} (u+v) \p_{t}\varphi \diff x\diff t  + \int_{\R^{d}}  (u^{0}+v^0)\, \varphi(0) \diff x - \int_{0}^{T}\int_{\R^{d}} (u+v)\,\nabla p\cdot\nabla\varphi\diff x \diff t = \\ =  \int_{0}^{T}\int_{\R^{d}} u\, F(p)\diff x\diff t +  \int_{0}^{T}\int_{\R^{d}} v\, G(p)\diff x\diff t,
\end{split}
\end{equation}
which implies continuity of $t \mapsto \int_{\R^d} (u+v)\, \psi(x) \diff x$. \\

\noindent The argument above shows that the sequence $
\{(u_{\sigma}+ v_{\sigma})(T)\}$ is bounded in $L^{\gamma}(\R^d)$ so it has a subsequence converging to some $s \in L^{\gamma}(\R^d)$ when $\sigma \to 0$. We claim that $s = (u+ v)(T)$. To this end, we consider the weak formulation from Definition \ref{def:Brinkman} with test function of the form $\psi(x)\, \eta_{\delta}(t)$ where $\eta_{\delta} = 1$ on $[0, T-2\delta]$, $\eta_{\delta} = 0$ on $[T-\delta, T]$ and it is linear interpolation on $[T-2\delta, T-\delta]$
(such function is admissible by density as it has Sobolev derivative). By weak continuity, as $\delta \to 0$,
\begin{multline*}
\int_0^T \int_{\R^d} (u_{\sigma}+v_{\sigma})(t,x) \, \psi(x) \partial_t \eta_{\delta}(t) \diff x \diff t = \frac{1}{\delta} \int_{T-2\delta}^{T-\delta} \int_{\R^d} (u_{\sigma} + v_{\sigma})(t,x)\,  \psi(x) \diff x \diff t  \\ 
\xrightarrow{\delta\to0} \int_{\R^d} (u_{\sigma} + v_{\sigma})(T,x)\,  \psi(x) \diff x \xrightarrow{\sigma \to 0} \int_{\R^d} s(x)\,\psi(x) \diff x. 
\end{multline*}
We can apply the same limiting procedure $\delta \to 0$ in the weak formulation \eqref{eq:weak_form_the_limit_sum}. By comparing the results, we deduce that $s = (u + v)(T)$ a.e. on $\R^d$ and this concludes the proof.
\end{proof}

\section{Strong convergence of $\nabla W_{\sigma}$ and conclusion}
\begin{proof}[Proof of Theorem \ref{thm:brinkman_to_darcy}]
Due to the weak convergence of $u_{\sigma}$ and $v_{\sigma}$, cf. \eqref{weak_u}--\eqref{weak_v}, to pass to the limit in \eqref{eq:Brinkman}, it is sufficient to prove strong convergence of $\nabla W_\sigma$. As $\nabla W_{\sigma} \weak \nabla p$, it is sufficient to prove convergence of $L^2$ norms, i.e.
$$
\| \nabla W_{\sigma} \|^2_{L^{2}((0,T)\times\R^{d})} \to \| \nabla p \|^2_{L^{2}((0,T)\times\R^{d})}.  
$$
By the properties of the weak convergence
\begin{equation}\label{liminf_estimate}
\left\|\nabla p\right\|^2_{L^{2}((0,T)\times\R^{d})}\le \liminf_{\sigma\to 0}\left\|\nabla W_{\sigma}\right\|^2_{L^{2}((0,T)\times\R^{d})}.
\end{equation}
so we only need to estimate $\limsup_{\sigma \to 0}$. The idea is to pass to the limit in Equation~\eqref{eq:sum}. First, due to the strong compactness of $u+v$ and $p$ in Proposition~\ref{prop:convergence}, we can pass to the limit in Equation~\eqref{eq:sum} and obtain (in the weak sense)
\begin{equation}\label{eq:sum2}
\p_{t}(u+v) - \DIV((u+v)\nabla p ) = u \, F(p) + v\, G(p).  
\end{equation}

We can test this equation with $\gamma (u+v)^{\gamma-1}$ (see Remark \ref{rem:rigorous_integration_by_parts_main_result} below for the precise argument) and we obtain after integrating in time
\begin{equation}\label{eq:mass_pressure_limit}
\begin{split}
\int_{\R^d} p(T,x) \diff x +  ( \gamma-1) \int_0^T \int_{\R^d}  |\nabla p|^2\diff x \diff t &= \\ = \int_{\R^d} p^0(x) \diff x +  \gamma \int_0^T \int_{\R^d}& (u\,F(p) + v\,G(p)) (u+v)^{\gamma-1} \diff x \diff t.
\end{split} 
\end{equation}
Note that \eqref{eq:weak_conv_trace} implies
\begin{equation}\label{eq:liminf_for_the_trace}
\int_{\R^d} p(T,x) \diff x \leq \liminf_{\sigma \to 0} \int_{\R^d} p_{\sigma}(T,x) \diff x
\end{equation}
by the weak lower semicontinuity of the norm. Integrating \eqref{eq:mass_pressure} in time from $[0,T]$ and applying $\limsup_{\sigma\to0}$ we see that 
\begin{equation}\label{eq:mass_pressure_limit2}
\begin{split}
( \gamma-1)\limsup_{\sigma\to 0} \int_0^T \int_{\R^d}  |\nabla W_\sigma|^2\diff x \diff t \le 
\int_{\R^d} p^0(x) \diff x &\,+ \\
+\, \gamma \int_0^T \int_{\R^d} (u\,F(p) + v\,G(p)) (u+v)^{\gamma-1} \diff x \diff t\, - &\liminf_{\sigma \to 0} \int_{\R^d} p_{\sigma}(T,x) \diff x.
\end{split}
\end{equation}
Combining~\eqref{eq:mass_pressure_limit}, \eqref{eq:liminf_for_the_trace} and~\eqref{eq:mass_pressure_limit2} we see that 
\begin{equation}\label{limsup_estimate}
\limsup_{\sigma\to 0}\left\|\nabla W_{\sigma}\right\|^2_{L^{2}((0,T)\times\R^{d})}\le \left\|\nabla p\right\|^2_{L^{2}((0,T)\times\R^{d})}.
\end{equation}

which together with~\eqref{liminf_estimate} concludes the proof.
\end{proof}
\begin{rem}\label{rem:rigorous_integration_by_parts_main_result} To make integration by parts rigorous, we test equation with $$
    \gamma (u+v)^{\gamma-1}\, \psi_R(x)
    $$ 
    where $\psi_R$ is a smooth function such that $\psi_R(x) = 1$ for $|x|\leq R$, $\psi_R(x) = 0$ for $|x| \geq R+1$ and $|\psi_R'|\leq 1$. Then, the integration by parts is justified and we obtain
    $$
    (\gamma-1)\, 
    \int_0^T \int_{\R^d} |\nabla p|^2 \, \psi_R(x) \diff x \diff t + \gamma\, \int_0^T \int_{R\leq |x|\leq R+1} p\, \nabla p\, \psi' \diff x \diff t .
    $$
    The first term converges to $(\gamma-1)\, \int_{\R^d} |\nabla p|^2 \, \psi_R(x) \diff x$ by the dominated convergence theorem. For the second term, we note that $p\,\nabla p \in L^1((0,T)\times\R^d)$ (in fact, we have even better) so this term converges to 0 again by the dominated convergence theorem. 
\end{rem}
\begin{rem}
Similarly, to make testing \eqref{eq:sum2} with $\gamma\,(u+v)^{\gamma-1}$ rigorous, one mollifies \eqref{eq:sum2} both in time and space with $\eta_{\delta}(t)\, \psi_{\eps}(x)$. Then, one tests \eqref{eq:sum2} with $\gamma\,((u+v)\ast \eta_{\delta} \ast  \psi_{\eps})^{\gamma-1}$ so that the usual chain rule in Sobolev spaces can be applied resulting in the term of the form $\int_{\R^d} ((u+v)\ast \eta_{\delta} \ast  \psi_{\eps})^{\gamma}(T, x)\diff x$. Then, one sends $\delta \to 0$ using weak continuity of the sum $(u+v)$ as in the proof of Proposition \ref{prop:convergence} and then $\eps \to 0$ using the properties of the mollifiers.
\end{rem}

\appendix
\section{Proof of Theorem \ref{thm_existence_of_solutions} (existence result)}\label{sect:existence_solutions}
Here, we prove existence of solutions to the system \eqref{eq:Brinkman} by introducing an artificial diffusion as in \cite{david2023degenerate} for the case $\gamma =1$. We also recall the uniform (in terms of $\gamma$ and $\sigma$) $L^{\infty}$ bounds from \cite{MR3037041}.\\

We rewrite \eqref{eq:Brinkman} as follows
\begin{equation*}
\partial_t u -  
\nabla u \cdot  \nabla W - \frac{1}{\sigma} \, u\, (W-p) =  u\,F(p), \quad 
\partial_t v -  
\nabla v \cdot  \nabla W - \frac{1}{\sigma} \, v\, (W-p) =  v\,G(p),
\end{equation*}
where we skipped the lower index $\sigma$ as $\sigma$ is fixed. It is useful to write $W = K_{\sigma} \ast p$ where $K_{\sigma}$ is a fundamental solution $-\sigma \Delta K_{\sigma} + K_{\sigma} = \delta_0$ and $p = (u+v)^{\gamma}$. We regularize the problem in three ways. First, we introduce diffusion. Second, we mollify $K_{\sigma}$ with a usual mollifier $\omega_{\delta}$ (i.e. $\omega_{\delta}(x) = \frac{1}{\delta^d} \omega(x/\delta)$ where $\omega$ is smooth, supported in the unit ball and of mass 1). Third, we truncate all nonlinearities by the truncation operator $Q(p) = p\, \mathds{1}_{p \leq 2 p_H} + 2\,p_{H}\, \mathds{1}_{p > 2p_{H}}$. The resulting system reads
\begin{equation}\label{eq:approximating_solution}
    \begin{split}
\partial_t u - \varepsilon\Delta u- 
\nabla u \cdot  \nabla W_{\delta} &= \frac{1}{\sigma} \, u\, (W_{\delta}-Q(p)) +  u\,F(Q(p)),\\
\partial_t v - \varepsilon\Delta v- 
\nabla v \cdot  \nabla W_{\delta} &= \frac{1}{\sigma} \, v\, (W_{\delta}-Q(p)) +  v\,G(Q(p)),\\
W_{\delta} &= K_{\sigma}\ast \omega_{\delta}\ast Q(p),
    \end{split}
\end{equation}
where $p = (u+v)^{\gamma}$. By properties of convolutions, $\nabla W_{\delta} = K_{\sigma}\ast \nabla \omega_{\delta}\ast Q(p)$ so that \eqref{eq:approximating_solution} can be considered as a semilinear parabolic system with Lipschitz nonlinearities which are well-understood \cite{ladyzhenskaya1968linear,MR1465184}. In particular, the solutions can be constructed by the fixed point argument combined with Schauder's estimates. We conclude that \eqref{eq:approximating_solution} has a nonnegative solution $(u,v)$.\\

Now, following \cite[Lemma 2.1]{MR3037041}, we claim that $p \leq p_H$. To this end, we sum up equations for $u$, $v$ and multiply by $\gamma\, (u+v)^{\gamma-1}$ to obtain
\begin{equation}\label{eq:equation_for_pressure_approx}
\begin{split}
\partial_t p - \varepsilon \, \gamma\,  \Delta (u+v) \, &(u+v)^{\gamma-1}  - \nabla p \cdot \nabla W_{\delta} = \\ &= \frac{\gamma}{\sigma} p\,(W_{\delta} - Q(p)) + \gamma\, (u+v)^{\gamma-1}\, (u\, F(Q(p)) + v\,G(Q(p))).
\end{split}
\end{equation}
Since at $t =0$, $p \leq p_H$, if the estimate is not satisfied, by continuity there is time $t>0$ where $p(t,\cdot)$ reaches its maximum with value in $(p_H, 2\, p_H)$. Let $x^*$ be a point where this happens. At this point, $\nabla p = 0$, $\Delta (u+v) \leq 0$ (because the function $u+v = p^{1/\gamma}$ reaches its maximum). Furthermore, the source term is strictly negative due to \ref{ass:item1} in Assumption \ref{ass_main}. Finally, we note that by Young's convolutional inequality
\begin{equation*}
W_{\delta}(t,x^*) - p(t,x^*) \leq \|p(t,\cdot)\|_{\infty} \, \|K_{\sigma}\|_{1}\, \|\omega_{\delta}\|_{1} - \|p(t,\cdot)\|_{\infty} = 0
\end{equation*}
so that the term $W_{\delta} - p$ at $x^*$ is nonpositive. We conclude that
$$
\partial_t p(t, x^*)<0
$$
so that $p$ cannot become greater than $p_H$. We conclude that $Q(p) = p$ in \eqref{eq:approximating_solution}.\\

Now, we send $\delta \to 0$. We write $u_{\delta}$ and $v_{\delta}$ for solutions to \eqref{eq:approximating_solution}, $p_{\delta} = (u_{\delta}+v_{\delta})^{\gamma}$ for the pressure and $W_{\delta} = K_{\sigma} \ast \omega_{\delta} \ast p_{\delta}$. First, thanks to the presence of diffusion and $\nabla K_{\sigma} \in L^1(\R^d)$, the sequences $\{u_{\delta}\}$ and $\{v_{\delta}\}$ are locally compact in $L^q((0,T)\times\R^d)$ for all $q <\infty$ by usual Lions-Aubin lemma and interpolation in Lebesgue spaces. As a consequence, it is easy to pass to the limit in the source terms. The nontrivial part is to pass to the limit in the advection term. To identify the limit we write
$$
\int_0^T \int_{\R^d} \nabla u_{\delta} \cdot \nabla W_{\delta} \, \varphi =
- \int_0^T \int_{\R^d} u_{\delta} \, \Delta W_{\delta}  \, \varphi + u_{\delta} \, \nabla W_{\delta}\cdot \nabla \varphi.  
$$
Hence, it is sufficient to prove that $\Delta W_{\delta}$ and $\nabla W_{\delta}$ converge at least weakly to the appropriate limits. However, these sequences are bounded in $L^{\infty}((0,T)\times\R^d)$ because $\Delta W_{\delta} = \frac{1}{\sigma}(p_{\delta} - p_{\delta}\ast K_{\sigma})\ast\omega_{\delta}$ and $\nabla W_{\delta} = \nabla K_{\sigma} \ast p_{\delta} \ast \omega_{\delta}$. Therefore, up to a subsequence, they have weak$^*$ limits which equals $\Delta W$ and $\nabla W$ with $W = K_{\sigma} \ast p$ due to the strong compactness of $u_{\delta}$ and $v_{\delta}$.\\

In the limit $\delta \to 0$, we obtain the system 
\begin{equation}\label{eq:approximation_last_step}
\partial_t u - \varepsilon\Delta u- 
\DIV(u \,  \nabla W) =   u\,F(p), \qquad 
\partial_t v - \varepsilon\Delta v- 
\DIV(v \,  \nabla W) =   v\,G(p),
\end{equation}    
where $W = K_{\sigma} \ast p$ and it remains to remove the diffusion, i.e. send $\varepsilon \to 0$. Again, we write $u_{\eps}$ and $v_{\eps}$ for solutions to \eqref{eq:approximation_last_step}, $p_{\eps} = (u_{\eps}+v_{\eps})^{\gamma}$ for the pressure and $W_{\eps} = K_{\sigma}\ast p_{\eps}$. Clearly, $u_{\eps} \weaks u$ and $v_{\eps} \weaks v$ in $L^{\infty}((0,T)\times\R^d)$. Moreover, standard computations show that $\{\partial_t u_{\eps}\}$, $\{\partial_t v_{\eps}\}$ are uniformly bounded in $L^2(0,T; H^{-1}(\R^d))$. We prove that both sequences $\{u_{\eps}\}$ and $\{v_{\eps}\}$ are strongly compact in space so that by Lemma \ref{lem:Jakub}, we deduce strong compactness. The same will follow for the pressure $p_{\eps}$ so that $\nabla W_{\eps} = \nabla K_{\sigma} \ast p_{\eps}$ converges in $L^1((0,T)\times\R^d)$ to $\nabla K_{\sigma} \ast p$ and so, by interpolation, in $L^q((0,T)\times\R^d)$ for all $q \in [1,2)$. This is sufficient to pass to the limit in \eqref{eq:approximation_last_step}.\\

The proof of compactness in space follows the method of Jabin and Belgacem \cite{MR2995703} (the only difference is that we deal with an additional source term). Let us recall that \cite{MR2995703} deals with compactness for the conservative equations
$$
\partial_t u_{\eps} - \varepsilon \Delta u_{\eps} - \DIV(u_{\eps}\,a_{\eps}) = 0,
$$
where $a_{\eps}$ is the vector field satisfying the following:
\begin{enumerate}
    \item $\sup_{\varepsilon \in (0,1)} \| \DIV a_{\eps} \|_{L^{\infty}((0,T)\times\R^d)} < \infty$,
    \item $\sup_{\varepsilon \in (0,1)}  \|a_{\eps}\|_{L^{\infty}(0,T; W^{1,p}(\R^d))} < \infty$ for some $p > 1$, 
    \item $\DIV a_{\eps} = d_{\eps} + r_{\eps}$ where $d_{\eps}$ is compact in space while $r_{\eps}$ is such that $|r_{\eps}(x) - r_{\eps}(y)| \leq C\, |u_{\eps}(x) - u_{\eps}(y)|$.
\end{enumerate}
In our case, $a_{\eps} = \nabla W_{\eps}$ satisfies (1) and (2). Indeed, $\DIV \nabla W_{\eps} = \Delta W_{\eps} = \frac{1}{\sigma} \left( p_{\eps} - K_{\sigma} \ast p_{\eps}\right) $ is uniformly bounded. Furthemore, it is easy to see, for instance from \eqref{eq:equation_for_pressure_approx}, that $\{p_{\eps}\}$ is uniformly bounded in $L^{\infty}(0,T; L^1(\R^d))$ so that it is bounded in $L^{\infty}(0,T; L^2(\R^d))$. Hence, $\{\Delta W_{\eps}\}$ is bounded in $L^{\infty}(0,T; L^2(\R^d))$ which easily implies that $\{W_{\eps}\}$ is bounded in $L^{\infty}(0,T; W^{2,2}(\R^d))$ so that (2) holds true with $p=2$.\\

Concerning (3), it is satisfied in a weaker sense. We have $d_{\eps} = W_{\eps}$ (it is compact in space by the estimate on $\{\nabla W_{\eps}\}$) and $r_{\eps} = p_{\eps}$ which satisfies (by the uniform boundedness of $\{u_{\eps}\}$ and $\{v_{\eps}\}$)
\begin{equation}\label{eq:Lipschitz_estimate_r_our_case}
|r_{\eps}(t,x) - r_{\eps}(t,y)| \leq C\, |u_{\eps}(t,x) - u_{\eps}(t,y)| + C\, |v_{\eps}(t,x) - v_{\eps}(t,y)|,
\end{equation}
so that the estimate depends on both species. Below, we briefly explain a simple modification of argument in \cite{MR2995703} to cover the case of \eqref{eq:Lipschitz_estimate_r_our_case} as well as how to include the source terms. \\

The compactness in \cite{MR2995703} is established by analysis of the quantity
$$
\mathcal{Q}_{u_{\eps}}(t) =  \int_{\R^d} \int_{\R^d} \mathcal{K}_{h}(x-y)|u_{\eps}(t,x)-u_{\eps}(t,y)|\diff x \diff y, 
$$
where $\mathcal{K}_h$ is a smooth, nonnegative kernel, supported in the ball $B_2(0)$ such that $\mathcal{K}_h(x) = \frac{1}{{(|x|^2 + h^2)}^{d/2}}$ on $B_1(0)$. Similarly, we define $\mathcal{Q}_{v_{\eps}}(t)$. It can be proved, cf. \cite[Lemma 3.1]{MR2995703}, that the sequence $\{u_{\eps}\}$ is locally compact in space in $L^1((0,T)\times \R^d)$ if
\begin{equation}\label{eq:criterion_cmpactness}
\lim_{h\to 0} \limsup_{\eps \to 0} \frac{1}{|\log h|}  \int_0^T \mathcal{Q}_{u_{\eps}}(t) \diff t = 0.
\end{equation}
Then, one computes $\frac{\diff}{\diff t} \mathcal{Q}_{u_{\eps}}$ using the PDE on $u_{\eps}$. Applying \cite[proof of Theorem 1.2]{MR2995703} with new assumption \eqref{eq:Lipschitz_estimate_r_our_case} and additional source term we deduce
\begin{equation}\label{eq:compactness_u_eps_operator_Q}
\begin{split}
\mathcal{Q}_{u_{\eps}}(t) \leq \, &\mathcal{Q}_{u_{\eps}}(0) + C\, \int_0^t \left(\mathcal{Q}_{u_{\eps}}(s) + \mathcal{Q}_{v_{\eps}}(s)\right) \diff s + C\, \frac{\varepsilon}{h^2} \\&+ 
\int_0^t \int_{\R^d} \int_{\R^d} \mathcal{K}_h(x-y)|d_{\eps}(t,x) - d_{\eps}(t,y)|\diff x \diff y \diff s
\\
& + \int_0^t \int_{\R^d} \int_{\R^d} \mathcal{K}_h(x-y) |u_{\eps}(x) F(p_{\eps}(x)) - u_{\eps}(y) F(p_{\eps}(y))| \diff x \diff y \diff s.
\end{split}
\end{equation}
The last integral can be bounded by $C\, \int_0^t \left(\mathcal{Q}_{u_{\eps}}(s) + \mathcal{Q}_{v_{\eps}}(s)\right) \diff s$. Now, to deduce compactness, we write the same expression as \eqref{eq:compactness_u_eps_operator_Q} for $\mathcal{Q}_{v_{\eps}}(t)$ and we sum up to deduce 
\begin{align*}
\mathcal{Q}_{u_{\eps}}(t) + \mathcal{Q}_{v_{\eps}}(t)\leq \, &\mathcal{Q}_{u_{\eps}}(0) + \mathcal{Q}_{v_{\eps}}(0)+ C\, \int_0^t \left(\mathcal{Q}_{u_{\eps}}(s) + \mathcal{Q}_{v_{\eps}}(s)\right) \diff s + C\, \frac{\varepsilon}{h^2} \\&+ 
\int_0^t \int_{\R^d} \int_{\R^d} \mathcal{K}_h(x-y)|d_{\eps}(t,x) - d_{\eps}(t,y)|\diff x \diff y \diff s.
\end{align*}
Now, it is easy to see that since $\{\nabla d_{\eps}\}$ is uniformly bounded in $L^1((0,T)\times\R^d)$, the last term is bounded by a constant. Hence, Gronwall's inequality and \eqref{eq:criterion_cmpactness} imply compactness in space of $\{u_{\eps}\}$ and $\{v_{\eps}\}$. 
\bibliographystyle{abbrv}
\bibliography{fastlimit}

\end{document}